\documentclass[12pt]{amsart}

\usepackage{amssymb}
\usepackage{hyperref} 
\usepackage{enumitem}

\usepackage{amsmath}
\usepackage{latexsym}
\usepackage{extarrows}
\usepackage{txfonts}
\usepackage{mathtools}
\usepackage{bm}
\usepackage{tikz-cd}
\usepackage{xcolor}

\makeatletter
\@namedef{subjclassname@2020}{%
  \textup{2020} Mathematics Subject Classification}
\makeatother

\usepackage[T1]{fontenc}

\theoremstyle{plain}
\newtheorem{theorem}{Theorem}[section]
\newtheorem{lemma}[theorem]{Lemma}

\newtheorem{claim}[theorem]{Claim}

\theoremstyle{definition}
\newtheorem{definition}[theorem]{Definition}

\newtheorem{example}[theorem]{Example}

\theoremstyle{remark}
\newtheorem{remark}[theorem]{Remark}

\newcommand{\Hom}{\mathrm{Hom}}

\newcommand{\id}{\mathrm{id}}
\newcommand{\supp}{\mathrm{supp}}
\newcommand{\cB}{\mathcal{B}}
\newcommand{\cA}{\mathcal{A}}
\newcommand{\Z}{\mathbb{Z}}

\newcommand{\N}{\mathbb{N}}
\newcommand{\E}{\mathbb E}
\newcommand{\C}{\mathcal C}

\newcommand{\I}{\mathcal I}
\newcommand{\Zk}{\mathrm Z}
\newcommand{\norm}[1]{\left\lVert #1\right\rVert}
\newcommand{\Prob}{\mathcal{P}} 
\newcommand{\pr}{\mathrm{pr}}

\title[Non-ergodic structure theorem]{The non-ergodic Host--Kra--Ziegler structure theorem for $\Z^d$-actions via measurable selections}

\author[A. Jamneshan]{Asgar Jamneshan}
\address{Faculty of Mathematics, TU Dresden, 01062 Dresden, Germany}
\email{asgar.jamneshan@tu-dresden.de}

\author[Machado]{Simon Machado}
 \address{ETH Zurich, 
 R\"{a}mistrasse 101, 
 8006 Zurich, 
 Switzerland}
\email{smachado@ethz.ch}



\begin{document}

\begin{abstract}
We establish a non-ergodic version of the Host--Kra--Ziegler structure theorem for measure-preserving $\Z^d$-actions. Our argument reduces the non-ergodic case to the ergodic theorem (for $d\ge 2$ due to Candela and Szegedy) via a measurable selection procedure. We also establish a non-ergodic vertical nilcharacter version of our main result. 
The non-ergodic version of the Host--Kra--Ziegler structure theorem is a key input in the companion paper by the second author and Fraczyk classifying point processes (i.e. random subsets) of $\mathbb{Z}^d$ whose law is invariant under the group $\mathrm{ASL}_d(\mathbb{Z})$ of affine transformations.
\end{abstract}

\maketitle

\baselineskip=17pt

\section{Introduction}

Host and Kra \cite{hk} and, independently, Ziegler \cite{z} established the following remarkable structure theorem for ergodic $\Z$-actions\footnote{Host and Kra \cite{hk} developed a theory of cubical structures and the corresponding seminorms. They used this framework to define systems of order $k$ and proved a structure theorem showing that ergodic systems of order $k$ are precisely $k$-step pro-nilsystems. This led to the identification of characteristic factors for multiple ergodic averages and to a proof of their convergence in $L^2$. Independently, Ziegler \cite{z} obtained the $L^2$-convergence of the same averages by proving that the universal characteristic factors she constructed are pro-nilsystems, via a different approach. Leibman \cite{leibman-factor} later clarified the connection between these two constructions, showing that the Host--Kra factors of order $k$ coincide with Ziegler's $k$th universal characteristic factors.}. The corresponding conclusion for ergodic $\Z^d$-actions follows from a more general theory developed in \cite{cz}.

\begin{theorem}\label{t:hkz}
Let $\mathrm{X}=(X,\nu,T)$ be an ergodic measure-preserving $\Z^d$-system and let $k\ge 1$.
Then the $k$th Host--Kra factor $\Zk_k(\mathrm{X})$ of $\mathrm{X}$ is (measure-theoretically) isomorphic to an inverse limit of $k$-step nilsystems.
\end{theorem}

Informally, the factor $\Zk_k(\mathrm{X})$ governs the distribution of $(k+1)$ dimensional cubes in $\mathrm{X}$, and thereby controls the limiting behaviour of $k$-linear ergodic averages. Its structural description in geometric terms plays an important role in the understanding of additive patterns such as arithmetic progressions in arithmetic combinatorics. The formal definitions of the Host--Kra factors and of $k$-step nilsystems, and the notation appearing in Theorem \ref{t:main} and Theorem \ref{t:vertical_nilcharacters} are given in the preliminaries in  Section~\ref{s:prelim}. 

Motivated by applications, the aim of this note is to remove the ergodicity assumption from Theorem~\ref{t:hkz}. Our main result is the following non-ergodic variant.

\begin{theorem}\label{t:main}
Let $\mathrm{X}=(X,\nu,T)$ be a not necessarily ergodic measure-preserving $\Z^d$-system and let $k\ge 1$.
Then there exists an inverse system of factors
\[
 \mathrm{Y}_1 \twoheadleftarrow \mathrm{Y}_2 \twoheadleftarrow  \cdots \twoheadleftarrow  \mathrm{X}
\]
such that:
\begin{itemize}
\item[(i)] each $\mathrm{Y}_n$ is a bundle of $k$-step nilsystems over the invariant factor $\Omega$, and the bundle projection $\mathrm{Y}_n\to\Omega$ agrees with the projection on the invariant factor;
\item[(ii)] $\Zk_k(\mathrm{X})$ is (measure-theoretically) isomorphic to the inverse limit of the $\mathrm{Y}_n$.
\end{itemize}
\end{theorem}

In a companion paper \cite{FraczykMachado}, Fraczyk and the second author leverage Theorem \ref{t:main} to classify random subsets of $\Z^d$ whose law is invariant under the affine group $\mathrm{ASL}_d(\mathbb{Z}) = \mathrm{SL}_d(\mathbb{Z}) \ltimes \mathbb{Z}^d$. This classification is driven by considerations in mathematical physics due to Marklof \cite[(20.2)]{Marklof}, alongside deep structural motivations in dynamics, including exchangeability \cite{de1937prevision} and discrete analogues of the Howe--Moore phenomenon \cite{HoweMoore}. The analysis in \cite{FraczykMachado} hinges on a two-step reduction where Theorem \ref{t:main} plays a pivotal role. Initially, invariance under upper triangular matrices in $\mathrm{SL}_d(\mathbb{Z})$ reduces the classification to a study of non-conventional ergodic averages, posing a structural problem for $\Z^d$-systems of order $k$. Methods from homogeneous dynamics are then deployed to resolve the resulting questions on bundles of nilsystems under an $\mathrm{ASL}_d(\mathbb{Z})$-action. Theorem \ref{t:main} acts as the crucial bridge connecting these two distinct dynamical environments.

A non-ergodic nilsequence formulation of Theorem \ref{t:hkz} for $\Z$-actions was established by Chu, Frantzikinakis, and Host, see \cite[Proposition 3.1]{chuetal}. This result has served as a key input in several ergodic-theoretic and arithmetic applications; see, e.g., \cite{f,fk,hsy,zk}. A version of Theorem~\ref{t:main} in the case $k=1$ for $\Z$-actions appears in \cite[Theorem 3.1]{edeko}. Closely related non-ergodic Mackey and Furstenberg--Zimmer structure theorems are developed in \cite{austin}; see also \cite{ejk}. 

We also record a non-ergodic version of the fact that ergodic systems of order $k$ are spanned by vertical nilcharacters. 

\begin{theorem}\label{t:vertical_nilcharacters}
Let $\mathrm X=(X,\nu,T)$ be a measure-preserving $\mathbb Z^d$-system and let $k\geq 1$. Then $L^2(\Zk_k(\mathrm X))$ is the closed linear span of the non-ergodic $k$-vertical nilcharacters on $\mathrm X$.
\end{theorem}

The case $k=1$ of Theorem \ref{t:vertical_nilcharacters} was established by Frantzikinakis and Host, see \cite[Theorem 5.2]{fh}, and has found numerous applications, see, e.g., \cite{abs,ar,fk,hernandez2026}  for some recent applications.

\subsection*{Acknowledgments}
The authors would like to thank Henrik Kreidler for helpful discussions. We thank Nikos Frantzikinakis for his feedback on earlier versions of the manuscript, in particular for suggesting Theorem \ref{t:vertical_nilcharacters}.  
 A.J. acknowledges support from the Deutsche Forschungsgemeinschaft (DFG, German Research Foundation) via the Heisenberg Grant -- 547294463 and its Excellence Strategy -- EXC-2047 -- 390685813. S.M. acknowledges support from a Hermann Weyl Instructorship at the Institute for Mathematical Research (FIM) at ETH Zurich. We thank the anonymous referee for their suggestions which helped improve the paper. 
\section{Preliminaries}\label{s:prelim}

In this section, we introduce our notation and collect some existing results that will be used in the proof of Theorem \ref{t:main}. We start by defining the Host--Kra factors via the device of Gowers--Host--Kra seminorms as introduced in \cite{hk}.  

Let $\mathrm{X}=(X,\nu,T)$ be a measure-preserving $\Z^d$-system formed on a standard Lebesgue probability space. 
Since $X$ is standard Borel, we fix once and for all a Polish topology on $X$ whose Borel $\sigma$-algebra coincides with the given one; this allows us to speak of the weak topology on spaces of probability measures on $X$ and products of $X$ with compact metric spaces.

We define measures $\nu^{[k]}$ on $X^{[k]}\coloneqq X^{\{0,1\}^k}$ inductively as follows.
Set $\nu^{[0]}:=\nu$ on $X^{[0]}:=X$.
For $k\ge 0$, define the diagonal $\Z^d$-action $T^{[k]}$ on $X^{[k]}$ by
\[
(T^{[k]})^{t}\bigl((x_\omega)_{\omega\in\{0,1\}^k}\bigr)
:=
\bigl(T^{t}x_\omega\bigr)_{\omega\in\{0,1\}^k},
\qquad t\in\Z^d.
\]
Let $\I^{[k]}$ be the $\sigma$-algebra of $T^{[k]}$-invariant sets in $X^{[k]}$.
Identify $X^{[k+1]}$ with $X^{[k]}\times X^{[k]}$ via
\[
(x_\omega)_{\omega\in\{0,1\}^{k+1}}
\longleftrightarrow
\Bigl( (x_{\omega 0})_{\omega\in\{0,1\}^k},\ (x_{\omega 1})_{\omega\in\{0,1\}^k} \Bigr).
\]
Then define $\nu^{[k+1]}$ to be the relatively independent joining of $\nu^{[k]}$ with itself
over $\I^{[k]}$, i.e.\ for all $f,g\in L^\infty(\nu^{[k]})$,
\[
\int_{X^{[k+1]}} f(x')\,g(x'')\,d\nu^{[k+1]}(x',x'')
:=
\int_{X^{[k]}} \E(f\mid \I^{[k]})\ \E(g\mid \I^{[k]})\ d\nu^{[k]}.
\]
For $f\in L^\infty(\nu)$ and $k\ge 1$, define
\[
\norm{f}_{U^k}^{2^k}
:=
\int_{X^{[k]}}
\prod_{\omega\in\{0,1\}^k}
\C^{|\omega|} f(x_\omega)\ d\nu^{[k]}((x_\omega)_\omega),
\]
where $\C$ denotes complex conjugation and $|\omega|=\omega_1+\cdots+\omega_k$. 

For each $k\ge 1$ there is a factor 
$\Zk_{k-1}=\Zk_{k-1}(\mathrm{X})$ of $\mathrm{X}$ characterized by the property that, for every
$f\in L^\infty(\nu)$,
\begin{equation}\label{eq:Zk-characterization}
\norm{f}_{U^k}=0
\quad\Longleftrightarrow\quad
\E(f\mid \Zk_{k-1})=0.
\end{equation}
We refer to $\Zk_{k}$ as the \emph{$k$th Host--Kra factor}.

We denote by 
\[
\pi_0:X\to \Omega
\]
the invariant factor, and write the ergodic decomposition as
\[
\nu=\int_\Omega \nu_\xi\, d\mu(\xi),
\]
where $\xi\mapsto \nu_\xi$ is a measurable probability kernel $\Omega\to \Prob(X)$ and $(X,\nu_\xi)$ is ergodic for $\mu$-a.e.\ $\xi$.

\begin{lemma}\label{lem:Zk_fiberwise}
Let $\mathrm{X}=(X,\nu,T)$ be a measure-preserving $\Z^d$-system and let $\nu=\int_\Omega \nu_\xi\, d\mu(\xi)$ be its ergodic decomposition. 
For all $k\geq 1$ and $\mu$-a.e.\ $\xi$ one has $\Zk_k(X,\nu_\xi, T)=(\Zk_k(\mathrm{X}),\nu_\xi, T)$. 
\end{lemma}
\begin{proof}
    See, e.g., \cite[Chapter 8, Proposition 18]{hkb}. 
\end{proof}

We will construct bundles of nilmanifolds over the invariant factor to geometrically represent Host--Kra factors of non-ergodic systems. 
The following definition formalizes this notion. 

\begin{definition}[Bundle of nilsystems]\label{Definition: Bundles of nilsystems}
    Let $(\Omega, \mu)$ be a standard Borel probability space. Choose:
\begin{enumerate}
    \item a measurable partition $(\Omega_n) _{n \geq 0}$ of $\Omega$;
    \item compactly generated nilpotent Lie groups $(N_n)_{n \geq 0}$ with discrete cocompact subgroups ($=$ lattices) $\Gamma_n \subset N_n$ for all $n \geq 0$;
    \item measurable maps $\phi_n: \Omega_n \rightarrow \Hom(\mathbb{Z}^d,N_n)$ such that for all $n \geq 0$ and $\mu$-almost all $\xi \in \Omega_n$ the $\mathbb{Z}^d$-action given by $\phi_n(\xi)$ on $(N_n/\Gamma_n, m_{N_n/\Gamma_n})$ is ergodic. 
\end{enumerate} 
Then the space 
\[
 X := \bigsqcup_{n \geq 0} \Omega_n \times N_n/\Gamma_n
\]
equipped with the probability measure
\[
\nu:=\sum_{n \geq 0} \mu_{\vert \Omega_n} \otimes m_{N_n/\Gamma_n}
\]
and the action defined for all $n \geq 0$, $t \in \mathbb{Z}^d$ and $(\xi,m) \in \Omega_n \times N_n/\Gamma_n$ by 
\[
 t \cdot (\xi,m) = (\xi, \phi_n(\xi)(t)\cdot m)
\]
is called a \emph{bundle of nilsystems (with base space $\Omega$)} and denoted by
\[
\bigsqcup_{n \geq 0} \Omega_n \rtimes_{\phi_n} N_n/\Gamma_n .
\]
We say that $X$ has \emph{step at most $k$} if $N_n$ is nilpotent of step at most $k$ for all $n \geq 0$. Finally, the obvious projection $X \rightarrow \Omega$ sends $\nu$ to $\mu$ and is called the \emph{bundle projection}.
\end{definition}

Next, we formalize our notion of non-ergodic vertical nilcharacter needed in Theorem \ref{t:vertical_nilcharacters}. 

For a nilpotent Lie group $N$, write
\[
N_{(1)}:=N,\qquad N_{(r+1)}:=[N,N_{(r)}].
\]
If $M=N/\Gamma$ is a nilmanifold of step at most $k$, then
\[
A_k(M):=N_{(k)}/(N_{(k)}\cap \Gamma)
\]
is a compact abelian Lie group. Since $N_{(k+1)}=\{e\}$, the group $N_{(k)}$ is central in $N$, and hence $A_k(M)$ acts on $M$ by left translations. If $N_{(k)}=\{e\}$, then $A_k(M)$ is understood to be the trivial group.

\begin{definition}[Vertical nilcharacters over the invariant factor]\label{def:vertical_nilcharacter_bundle}
Let
\[
\mathrm{Y}=\bigsqcup_{i\geq 0}\Omega_i\rtimes_{\alpha_i} N_i/\Gamma_i
\]
be a bundle of nilsystems of step at most $k$ over $(\Omega,\mu)$, in the sense of Definition~\ref{Definition: Bundles of nilsystems}. Put
\[
M_i:=N_i/\Gamma_i,\qquad A_i:=A_k(M_i)=(N_i)_{(k)}/((N_i)_{(k)}\cap \Gamma_i).
\]
A function $F\in L^2(Y)$ is called a \emph{$k$-vertical nilcharacter on $\mathrm Y$} if, for every $i\geq 0$, there exists a measurable map $\chi_i:\Omega_i\to \widehat{A_i}$ 
such that
\[
F(\xi,a\cdot m)=\chi_i(\xi)(a)\,F(\xi,m)
\]
for $\mu|_{\Omega_i}\otimes m_{A_i}\otimes m_{M_i}$-a.e.\ $(\xi,a,m)\in \Omega_i\times A_i\times M_i$. 

If $\mathrm X=(X,\nu,T)$ is a measure-preserving system, a function $f\in L^2(X)$ is called a \emph{non-ergodic $k$-vertical nilcharacter on $\mathrm X$} if there exist a factor map
$\pi:\mathrm X\to \mathrm Y$ onto a bundle $\mathrm{Y}$ of nilsystems of step at most $k$, and a $k$-vertical nilcharacter $F\in L^2(Y)$, such that $f=F\circ \pi$ in 
$L^2(X)$.   
\end{definition}

\begin{example}\label{rem:k_equals_one_vertical_nilcharacters}
Let
\[
\mathrm{Y}=\bigsqcup_{i\geq 0}\Omega_i\rtimes_{\alpha_i} M_i,
\qquad M_i=N_i/\Gamma_i,
\]
be a bundle of nilsystems of step at most $1$. Then each $N_i$ is abelian and
\[
A_1(M_i)=N_i/(N_i\cap \Gamma_i)=M_i.
\]
Thus a $1$-vertical nilcharacter on $\mathrm{Y}$ is a function
$F\in L^2(Y)$ such that, for every $i\geq 0$, there exists a measurable map $\chi_i:\Omega_i\to \widehat{M_i}$ 
with
\[
F(\xi,a\cdot m)=\chi_i(\xi)(a)F(\xi,m)
\]
for $\mu|_{\Omega_i}\otimes m_{M_i}\otimes m_{M_i}$-a.e.
$(\xi,a,m)\in\Omega_i\times M_i\times M_i$. Equivalently, for
$\mu$-a.e. $\xi\in\Omega_i$, the function $m\mapsto F(\xi,m)$ is a scalar
multiple of a character of the compact abelian Lie group $M_i$, where the
character is allowed to vary measurably with $\xi$.

In particular, if
\[
t\cdot(\xi,m)=(\xi,\alpha_i(\xi)(t)\cdot m),
\]
then
\[
F(t\cdot(\xi,m))
=
\lambda_t(\xi)F(\xi,m),
\qquad
\lambda_t(\xi):=\chi_i(\xi)\bigl(\alpha_i(\xi)(t)\Gamma_i\bigr).
\]
Thus, after pulling back to a system $\mathrm X$ through a factor map
$\pi\colon \mathrm X\to \mathrm Y$, a $1$-vertical nilcharacter is a relative eigenfunction over
the invariant factor $\Omega$, with invariant eigenvalue family
$(\lambda_t)_{t\in\Z^d}$. 

In the case $d=1$, after the normalization
$\E(|F|^2\mid\Omega)\in\{0,1\}$, this agrees with the notion of
non-ergodic eigenfunction used by Frantzikinakis and Host
\cite[Section~5.2]{fh}. 
\end{example}

\begin{remark}\label{rem:borelmod0}
In the arguments below, some partitions/maps are obtained only $\mu$-measurably (e.g.\ via the Jankov--von Neumann uniformization theorem).
Since $\Omega$ is standard Borel, after modifying on a $\mu$-null set one may arrange these to be Borel. We will freely work modulo null sets.
In particular, we will (after modifying on a $\mu$-null set) regard the kernel $\xi\mapsto \nu_\xi\in\Prob(X)$ as Borel.
\end{remark}

To perform our measurable selection argument, we need the following countability result for nilmanifolds.

\begin{theorem}\label{t:countability}
Up to isomorphism of nilmanifolds, there are only countably many pairs $(N,\Gamma)$ where $N$ 
is a compactly generated nilpotent Lie group and $\Gamma$ is a discrete cocompact subgroup.
\end{theorem}

\begin{proof}
See, e.g., \cite[Chapter 12, Theorem 28]{hkb}. 
\end{proof}

Our measurable selection argument relies on the Jankov--von Neumann uniformization theorem:

\begin{theorem}\label{thm:JvN}
Let $A\subseteq U\times V$ be a Borel subset of a product of standard Borel spaces.
If every section $A_u:=\{v:(u,v)\in A\}$ is nonempty for $u$ in a measurable set $U_0\subseteq U$,
then there exists a universally measurable map $s:U_0\to V$ with $(u,s(u))\in A$ for all $u\in U_0$.
\end{theorem}

\begin{proof}
See, e.g., \cite[Chapter 18]{ke}. 
\end{proof}

Finally, we need the following disintegration result of Kallenberg. 

\begin{lemma}\label{t:kallenberg}
Let $V,W$ be standard Borel spaces.  There exists a Borel map
\[
\Phi:\ V\times \Prob(V\times W)\times \Prob(V)\ \longrightarrow\ \Prob(W)
\]
such that whenever $\lambda\in\Prob(V\times W)$ and $\nu\in\Prob(V)$ satisfy $\lambda(\cdot\times W)\mathrel{=}\nu$,
one has the disintegration identity 
\[
\lambda(A\times B)=\int_A \Phi(s,\lambda,\nu)(B)\, d\nu(s)
\qquad \text{for all Borel }A\subseteq V,\ B\subseteq W. 
\]
In particular, for fixed $\nu$, the map $(s,\lambda)\mapsto \Phi(s,\lambda,\nu)$ is jointly Borel.
\end{lemma}

\begin{proof}
See \cite[Lemma 2.2]{ka}. Under the assumption $\lambda(\cdot\times W)=\nu$, the disintegration identity with $B=W$ forces $\Phi(s,\lambda,\nu)(W)=1$ for $\nu$-a.e.\ $s$. Fixing $w_0\in W$ and redefining $\Phi(s,\lambda,\nu):=\delta_{w_0}$ on the Borel set where $\Phi(s,\lambda,\nu)$ is not a probability measure, we may assume that $\Phi$ takes values in $\Prob(W)$; this does not affect the disintegration identity.
\end{proof}

\section{Proof of Theorem \ref{t:main}}

It suffices to prove Theorem~\ref{t:main} under the additional assumption that
\[
\mathrm{X}=\Zk_k(\mathrm{X}).
\]
Indeed, in the general case one applies the result below to the factor $\Zk_k(\mathrm{X})$ and then pulls back the resulting inverse system to $\mathrm{X}$. The invariant factor of $\Zk_k(\mathrm{X})$ is the same as the invariant factor of $\mathrm{X}$, since the invariant factor is contained in every Host--Kra factor.

We fix the following setting throughout the proof of this reduced case.  
\begin{itemize}
    \item $k\geq 1$. 
    \item A measure-preserving $\Z^d$-system $\mathrm{X}=(X,\nu,T)$ such that $\mathrm{X} =\Zk_k(\mathrm{X})$.  
    \item Let $\nu=\int_\Omega \nu_\xi\, d\mu(\xi)$ be the ergodic decomposition of $\mathrm{X}$. 
    \item A countable set $\mathcal D=\{f_r:r\in\mathbb N\}\subset L^\infty(\nu)$ dense in $L^2(\nu)$. 
    \item An enumeration of representatives of nilmanifolds of step at most $k$ (using Theorem \ref{t:countability}):
\[
(N_j,\Gamma_j)_{j\in\mathbb N}, \qquad M_j:=N_j/\Gamma_j.
\]
\item For each $j$, a compatible metric $d_j$ on $M_j$ and a countable set
$\mathcal U_j=\{u_{j,q}\}_{q\in\mathbb N}\subset C(M_j)$ dense in $L^2(m_{M_j})$. 
\item For each $j$, let $\Prob(X\times M_j)$ be the Polish space of Borel probability measures on $X\times M_j$ with the weak topology, and let 
\[
\mathsf Q:=\bigsqcup_{j\in\mathbb N}\Prob(X\times M_j)
\]
be the disjoint union standard Borel space of all these spaces of measures. 
\item For each $j$, let
\[
\mathsf P_j:=\Hom(\Z^d,N_j),
\]
be identified with the closed subset of $N_j^d$ consisting of commuting $d$-tuples (hence $\mathsf P_j$ is Polish), and let 
\[
\mathsf P:=\bigsqcup_{j\in\mathbb N}\mathsf P_j
\]
be the disjoint union standard Borel space of all such actions. 
\end{itemize}

\begin{lemma}\label{lem1}
For $\mu$-a.e.\ $\xi$ and for every $n\in\mathbb N$ there exist $j\in\mathbb N$,
an action $\alpha\in\mathsf P_j$, and a $k$-step nilfactor map
\[
\psi_{n,\xi}:(X,\nu_\xi)\to (M_j,m_{M_j})
\]
intertwining $T$ and $\alpha$, such that the action $\alpha$ on $(M_j,m_{M_j})$ is ergodic and such that for all $1\le r\le n$,
\[
\big\| f_r - \E_{\nu_\xi}(f_r \mid \psi_{n,\xi}^{-1}\cB(M_j))\big\|_{L^2(\nu_\xi)} < 2^{-(n+2)}.
\]
\end{lemma}

\begin{proof}
Fix $\xi$ such that $(X,\nu_\xi,T)$ is ergodic. By Lemma~\ref{lem:Zk_fiberwise} and Theorem \ref{t:hkz}, $(X,\nu_\xi,T)$ is an inverse limit of $k$-step nilsystems. Hence for the finite set
$\{f_1,\dots,f_n\}$ and tolerance $2^{-(n+2)}$, there is a $k$-step nilfactor $\psi_{n,\xi}$ such that each $f_r$
is within $2^{-(n+2)}$ in $L^2(\nu_\xi)$ of its conditional expectation onto that factor. By Theorem \ref{t:countability}
we may take the nilfactor isomorphic to some $M_j$ from our countable list, with action given by some $\alpha\in\mathsf P_j$.
Since this nilsystem is a factor of the ergodic system $(X,\nu_\xi,T)$, the action $\alpha$ on $(M_j,m_{M_j})$ is ergodic.
\end{proof}

\subsection*{Graph joinings and measurable reconstruction}

If $\phi:(X,\nu_\xi)\to (M_j,m_{M_j})$ is a factor map, define its \emph{graph joining}
\[
\lambda:=(\id_X,\phi)_*\nu_\xi=\int_X \delta_x\times \delta_{\phi(x)} \, d\nu_\xi(x) \in \Prob(X\times M_j).
\]
Then $(\pr_X)_*\lambda=\nu_\xi$ and $(\pr_{M_j})_*\lambda=m_{M_j}$.

Fix $j\in\N$. For $(\xi,\lambda,x)\in\Omega\times\Prob(X\times M_j)\times X$ with $(\pr_X)_*\lambda=\nu_\xi$, define the conditional measure
\[
\lambda_x := \Phi(x,\lambda,\nu_\xi)\in \Prob(M_j),
\]
where $\Phi$ is given by Lemma \ref{t:kallenberg}. By Remark \ref{rem:borelmod0}, the map $\xi\mapsto \nu_\xi$ is Borel, and hence
\[
(\xi,\lambda,x)\longmapsto \lambda_x
\]
is Borel on its natural domain $\{(\xi,\lambda,x):(\pr_X)_*\lambda=\nu_\xi\}$.

\begin{lemma}\label{lem2}
Fix $j\in\N$. Consider the Dirac embedding
\[
\delta:M_j\to \Prob(M_j),\qquad m\mapsto \delta_m.
\]
Then $\delta$ is continuous and injective, its image $\delta(M_j)\subseteq \Prob(M_j)$ is Borel, and the inverse
$\delta^{-1}:\delta(M_j)\to M_j$ is Borel.

Define the Borel set
\[
\mathcal G_j:=\{(\xi,\lambda,x): (\pr_X)_*\lambda=\nu_\xi \ \text{and}\  \lambda_x\in \delta(M_j)\}.
\]
On $\mathcal G_j$ define
\[
\Psi_j(\xi,\lambda,x):=\delta^{-1}(\lambda_x)\in M_j.
\]
Then $\Psi_j:\mathcal G_j\to M_j$ is Borel.

Moreover, for $(\xi,\lambda)$ with $(\pr_X)_*\lambda=\nu_\xi$ and $(\pr_{M_j})_*\lambda=m_{M_j}$, the following are equivalent:
\begin{itemize}
\item[(i)] $\lambda$ is supported on the graph of a measurable map $X\to M_j$;
\item[(ii)] $\lambda_x$ is a Dirac mass for $\nu_\xi$-a.e.\ $x$;
\item[(iii)] one has
\[
\int_X \left(\int_{M_j}\int_{M_j} d_j(m,m')\,d\lambda_x(m)\,d\lambda_x(m')\right)\,d\nu_\xi(x)=0.
\]
\end{itemize}
In particular, the set of pairs $(\xi,\lambda)$ with $(\pr_X)_*\lambda=\nu_\xi$, $(\pr_{M_j})_*\lambda=m_{M_j}$ and $\lambda$ a graph joining is Borel in $\Omega\times \Prob(X\times M_j)$.

Finally, whenever $(\xi,\lambda)$ is a graph joining, the map $x\mapsto \Psi_j(\xi,\lambda,x)$ is defined $\nu_\xi$-a.e.\ and satisfies
\[
(\id_X,\Psi_j(\xi,\lambda,\cdot))_*\nu_\xi=\lambda.
\]
We denote this a.e.\ defined map by $\phi_\lambda$.
\end{lemma}

\begin{proof}
The first paragraph follows from the Lusin--Souslin theorem (e.g.\ \cite[Theorem 15.1]{ke}) applied to the continuous injective map $\delta$.

Borelness of $\mathcal G_j$ and $\Psi_j$ follows because $(\xi,\lambda,x)\mapsto \lambda_x$ is Borel on $\{(\pr_X)_*\lambda=\nu_\xi\}$, the set $\delta(M_j)$ is Borel in $\Prob(M_j)$, and $\delta^{-1}$ is Borel on its domain.

For the equivalences: (i)$\Leftrightarrow$(ii) is the definition of a graph joining in terms of conditional measures. (ii)$\Leftrightarrow$(iii) uses the metric characterization of Dirac masses: for any probability measure $\eta$ on a metric space, $\eta$ is a Dirac mass if and only if $\int\!\!\int d(m,m')\,d\eta(m)\,d\eta(m')=0$. Apply this to $\eta=\lambda_x$ for $\nu_\xi$-a.e.\ $x$ and integrate over $x$.

For the final statement, if $\lambda_x=\delta_{\psi(x)}$ $\nu_\xi$-a.e., then on the $\nu_\xi$-conull set where $\lambda_x\in\delta(M_j)$ we have $\psi(x)=\delta^{-1}(\lambda_x)=\Psi_j(\xi,\lambda,x)$, hence
\[
(\id_X,\Psi_j(\xi,\lambda,\cdot))_*\nu_\xi=\lambda.
\]
\end{proof}

\subsection*{The selection relation}

Fix $n\in\mathbb N$ and define a relation $R_n\subseteq \Omega\times \mathsf P\times \mathsf Q$ by
\[
(\xi;\ (j,\alpha),\ \lambda)\in R_n
\]
if and only if all of the following hold:
\begin{itemize}
\item[(R1)]  $\lambda\in\Prob(X\times M_j)$ has marginals $(\pr_X)_*\lambda=\nu_\xi$, $(\pr_{M_j})_*\lambda=m_{M_j}$,
and $\lambda$ is a graph joining.
\item[(R2)] For every $t\in\Z^d$,
\[
(T^t\times \alpha(t))_*\lambda=\lambda.
\]
\item[(R3)]  For each $1\le r\le n$ there exists $q\in \N$ such that
\[
\int_{X\times M_j} \bigl|f_r(x)-u_{j,q}(m)\bigr|^2\,d\lambda(x,m) < 2^{-2n}.
\]
\item[(R4)] The action $\alpha$ on $(M_j,m_{M_j})$ is ergodic.
\end{itemize}

\begin{lemma}\label{lem:Rn_borel}
For each $n\in\mathbb N$, the set $R_n$ is a Borel subset of $\Omega \times \mathsf P \times \mathsf Q$. 
\end{lemma}

\begin{proof}
\noindent\emph{Borelness of (R2).}
Fix $t\in\Z^d$. For each $j$, the evaluation map
\[
\mathsf P_j\to N_j,\qquad \alpha\mapsto \alpha(t)
\]
is continuous (indeed, $\mathsf P_j$ is a closed subset of $N_j^d$ and $\alpha(t)$ is a word in the $d$ generators).
The action map $N_j\times M_j\to M_j$ is continuous, hence the map
\[
\Xi_{t,j}: \mathsf P_j\times (X\times M_j)\to X\times M_j,\qquad
(\alpha,(x,m))\mapsto (T^t x,\ \alpha(t)\cdot m)
\]
is Borel. Therefore the induced pushforward map
\[
(\alpha,\lambda)\ \longmapsto\ (\Xi_{t,j}(\alpha,\cdot))_*\lambda = (T^t\times \alpha(t))_*\lambda
\]
from $\mathsf P_j\times \Prob(X\times M_j)$ to $\Prob(X\times M_j)$ is Borel.
Consequently, the set
\[
A_{t}:=\bigl\{(\xi,(j,\alpha),\lambda):\ (T^t\times \alpha(t))_*\lambda=\lambda\bigr\}
\]
is Borel, since it is the preimage of the diagonal under
\[
(\xi,(j,\alpha),\lambda)\mapsto\bigl((T^t\times \alpha(t))_*\lambda,\lambda\bigr).
\]
Since $\Z^d$ is countable, $\bigcap_{t\in\Z^d}A_t$ is Borel, proving that (R2) is Borel.

\smallskip
\noindent\emph{Borelness of (R3).}
Fix $1\le r\le n$ and $q\in\N$. On the component indexed by $j$ define
\[
I_{r,q}(\xi,(j,\alpha),\lambda)
:=\int_{X\times M_j} \bigl|f_r(x)-u_{j,q}(m)\bigr|^2\,d\lambda(x,m).
\]
Because $(x,m)\mapsto |f_r(x)-u_{j,q}(m)|^2$ is bounded Borel and $\lambda\mapsto \int \varphi\,d\lambda$ is Borel for bounded Borel $\varphi$,
the map $I_{r,q}$ is Borel. Hence
\[
B_{r,q}:=\bigl\{(\xi,(j,\alpha),\lambda):\ I_{r,q}(\xi,(j,\alpha),\lambda)<2^{-2n}\bigr\}
\]
is Borel. Now the clause (R3) for a fixed $r$ is $\exists q\in\N$ with $I_{r,q}<2^{-2n}$, i.e.
\[
B_r:=\bigcup_{q\in\N}B_{r,q},
\]
a countable union of Borel sets, hence Borel. Intersecting over $r=1,\dots,n$ shows (R3) is Borel.

\smallskip
\noindent\emph{Borelness of (R4).}
Let
\[
F_N:=\{-N,\dots,N\}^d\subseteq \Z^d.
\]
For $u\in C(M_j)$ and $\alpha\in\mathsf P_j$, define
\[
A_N^\alpha u(m):=\frac{1}{|F_N|}\sum_{t\in F_N}u(\alpha(t)\cdot m).
\]
By the mean ergodic theorem, the action $\alpha$ is ergodic if and only if, for every $q\in\N$,
\[
\lim_{N\to\infty}
\int_{M_j}\left|A_N^\alpha u_{j,q}(m)-\int_{M_j}u_{j,q}\,dm_{M_j}\right|^2\,dm_{M_j}(m)=0.
\]
This equivalence uses that $\mathcal U_j$ is dense in $L^2(m_{M_j})$. For fixed $q$ and $N$, the displayed integral is a Borel, indeed continuous, function of $\alpha$. Therefore the set of ergodic $\alpha\in\mathsf P_j$ is Borel, being
\[
\bigcap_{q\in\N}\bigcap_{\ell\in\N}
\bigcup_{N_0\in\N}\bigcap_{N\geq N_0}
\left\{
\alpha:
\int_{M_j}\left|A_N^\alpha u_{j,q}-\int_{M_j}u_{j,q}\,dm_{M_j}\right|^2\,dm_{M_j}<\frac{1}{\ell}
\right\}.
\]
Hence (R4) is Borel.

\smallskip
\noindent\emph{Borelness of (R1).}
Fix $j$. The marginal maps
\[
\lambda\mapsto (\pr_X)_*\lambda \in \Prob(X),
\qquad
\lambda\mapsto (\pr_{M_j})_*\lambda \in \Prob(M_j)
\]
are continuous (hence Borel) for weak topologies. Since $\xi\mapsto \nu_\xi\in\Prob(X)$ is Borel (Remark~\ref{rem:borelmod0}),
the set of $(\xi,\lambda)$ with $(\pr_X)_*\lambda=\nu_\xi$ is Borel, and the set with $(\pr_{M_j})_*\lambda=m_{M_j}$ is also Borel.
Finally, the graph-joining clause is Borel by Lemma~\ref{lem2}. Hence (R1) is Borel. 
\end{proof}

\begin{lemma}\label{lem4}
$(R_n)_\xi$ is nonempty for $\mu$-a.e.\ $\xi$.
\end{lemma}

\begin{proof}
Fix $\xi$ such that $(X,\nu_\xi, T)$ is ergodic. By Lemma~\ref{lem1}, there is a $k$-step nilfactor
$\psi_{n,\xi}:(X,\nu_\xi)\to(M_j,m_{M_j})$ with action $\alpha\in\mathsf P_j$ such that each $f_r$ ($r\le n$) is within $2^{-(n+2)}$ in $L^2(\nu_\xi)$
of $\E_{\nu_\xi}(f_r\mid \psi_{n,\xi}^{-1}\cB(M_j))$, and such that $\alpha$ is ergodic on $(M_j,m_{M_j})$.

Let $\lambda=(\id,\psi_{n,\xi})_*\nu_\xi$. Then (R1), (R2), and (R4) hold.

For (R3): for each $r\le n$, since $\E_{\nu_\xi}(f_r\mid \psi_{n,\xi}^{-1}\cB(M_j))$ is of the form $g_r\circ\psi_{n,\xi}$ with $g_r\in L^2(M_j)$,
we have
\[
\norm{f_r-g_r\circ\psi_{n,\xi}}_{L^2(\nu_\xi)}<2^{-(n+2)}.
\]
Approximate $g_r$ in $L^2(M_j)$ by some $u_{j,q}$ from the dense set $\mathcal U_j$ so that
\[
\norm{g_r-u_{j,q}}_{L^2(M_j)}<2^{-(n+2)}.
\]
Then
\[
\norm{f_r-u_{j,q}\circ\psi_{n,\xi}}_{L^2(\nu_\xi)}
\le \norm{f_r-g_r\circ\psi_{n,\xi}}_{L^2(\nu_\xi)}+\norm{g_r-u_{j,q}}_{L^2(M_j)}
<2^{-(n+1)},
\]
and hence
\[
\int_{X\times M_j} \bigl|f_r(x)-u_{j,q}(m)\bigr|^2\,d\lambda(x,m)
=
\norm{f_r-u_{j,q}\circ\psi_{n,\xi}}_{L^2(\nu_\xi)}^2
<2^{-2n}.
\]
Thus (R3) holds for suitable choices of $q$ depending on $r$. Hence $(\xi;(j,\alpha),\lambda)\in R_n$.
\end{proof}

\subsection*{Constructing the bundle of nilsystems}

We can apply Theorem~\ref{thm:JvN} to $R_n$. After discarding a $\mu$-null set,
obtain a $\mu$-measurable selector
\[
s_n(\xi)=\bigl((j_n(\xi),\alpha_n(\xi)),\lambda_{n,\xi}\bigr)
\]
with $(\xi;s_n(\xi))\in R_n$ for $\mu$-a.e.\ $\xi$.

Let $\Omega_{n,j}:=\{\xi:j_n(\xi)=j\}$.
Define
\[
Y_n:=\bigsqcup_{j\in\N}(\Omega_{n,j}\times M_j),
\qquad
\nu_{Y_n}:=\sum_{j\in\N}\mu|_{\Omega_{n,j}}\otimes m_{M_j},
\]
and define the $\Z^d$-action on $Y_n$ by
\[
t\cdot(\xi,m):=(\xi,\alpha_n(\xi)(t)\cdot m).
\]
By (R4), this is a bundle of $k$-step nilsystems over $\Omega$ in the sense of Definition~\ref{Definition: Bundles of nilsystems}.

Define the factor map $\pi_n:X\to Y_n$ by
\[
\pi_n(x):=\bigl(\pi_0(x),\,\phi_{n,\pi_0(x)}(x)\bigr),
\]
where $\phi_{n,\xi}:=\phi_{\lambda_{n,\xi}}$ is extracted from $\lambda_{n,\xi}$ via Lemma~\ref{lem2} applied on the component $M_{j_n(\xi)}$.

To make $\pi_n$ everywhere-defined, fix for each $j$ a basepoint $m_{j,0}\in M_j$ and define $\phi_{n,\xi}(x):=m_{j,0}$
on the $\nu_\xi$-null set where $\phi_{\lambda_{n,\xi}}$ is not defined; this does not affect any of the measure-theoretic identities below.

For each $n$, the map $\pi_n$ is measurable, satisfies
\[
\pi_n(T^t x)=t\cdot \pi_n(x)\quad\text{for }\nu\text{-a.e.\ }x,\ \forall t\in\Z^d,
\]
and $(\pi_n)_*\nu=\nu_{Y_n}$.  In particular $(Y_n,\nu_{Y_n})$ is a factor of $(X,\nu)$ and the bundle projection
$Y_n\to\Omega$ equals $\pi_0$.

Indeed, $\xi\mapsto \lambda_{n,\xi}$ is $\mu$-measurable by construction, and Lemma~\ref{lem2} provides a jointly measurable realization
$(\xi,\lambda,x)\mapsto \phi_\lambda(x)$ on graph joinings. Hence $x\mapsto \pi_n(x)$ is measurable.

For $\mu$-a.e.\ $\xi$, (R2) says $(T^t\times\alpha_n(\xi)(t))_*\lambda_{n,\xi}=\lambda_{n,\xi}$ for all $t$.
Since $\lambda_{n,\xi}$ is a graph joining, this is equivalent to $\phi_{n,\xi}(T^t x)=\alpha_n(\xi)(t)\phi_{n,\xi}(x)$ for $\nu_\xi$-a.e.\ $x$.
Integrating over $\xi$ yields the stated equivariance $\nu$-a.e.\ on $X$.

Disintegrate
\[
(\pi_n)_*\nu=\int_\Omega (\pi_n)_*\nu_\xi\,d\mu(\xi).
\]
For $\mu$-a.e.\ $\xi$, (R1) gives $(\pr_{M_{j_n(\xi)}})_*\lambda_{n,\xi}=m_{M_{j_n(\xi)}}$ and $(\pr_X)_*\lambda_{n,\xi}=\nu_\xi$,
so $(\pi_n)_*\nu_\xi=\delta_\xi\otimes m_{M_{j_n(\xi)}}$. Integrating gives $\nu_{Y_n}$.
Finally, the bundle projection composed with $\pi_n$ is $\pi_0$ by definition.

The factors $Y_n$ need not be increasing in $n$, so we take finite joins.

Let $\cA_n:=\sigma(\pi_1,\dots,\pi_n)$, and let $Z_n$ be the corresponding factor of $X$.
Equivalently, $Z_n$ is the factor induced by
\[
\Pi_n:X\to \prod_{i=1}^n Y_i,\qquad
\Pi_n(x):=(\pi_1(x),\dots,\pi_n(x)).
\]

\begin{lemma}\label{lem:join_bundle}
For each $n$, the factor $Z_n$ is a bundle of $k$-step nilsystems over $\Omega$ with bundle projection induced by $\pi_0$.
\end{lemma}

\begin{proof}
For $\ell=1,\dots,n$, write
\[
Y_\ell=\bigsqcup_{i\geq 0}\Omega_i^{(\ell)}\rtimes_{\alpha_i^{(\ell)}} M_i,
\]
where $M_i=N_i/\Gamma_i$ is one of the fixed representatives. For a tuple $\mathbf i=(i_1,\dots,i_n)$, put
\[
\Omega_{\mathbf i}:=\bigcap_{\ell=1}^n \Omega_{i_\ell}^{(\ell)},\qquad
M_{\mathbf i}:=M_{i_1}\times\cdots\times M_{i_n},
\]
and
\[
N_{\mathbf i}:=N_{i_1}\times\cdots\times N_{i_n},
\qquad
\Gamma_{\mathbf i}:=\Gamma_{i_1}\times\cdots\times \Gamma_{i_n}.
\]
For $\xi\in\Omega_{\mathbf i}$, let
\[
\eta_\xi:=(\Pi_n)_*\nu_\xi\in\Prob(M_{\mathbf i})
\]
after the evident identification of the fiber over $\xi$ with $M_{\mathbf i}$. The map $\xi\mapsto\eta_\xi$ is Borel. Let
\[
\beta_\xi(t):=
\bigl(\alpha_{i_1}^{(1)}(\xi)(t),\dots,\alpha_{i_n}^{(n)}(\xi)(t)\bigr)
\in N_{\mathbf i}.
\]
Then $\beta_\xi\in\Hom(\Z^d,N_{\mathbf i})$ depends measurably on $\xi$, and $\eta_\xi$ is invariant and ergodic under the nilrotation action given by $\beta_\xi$. Ergodicity follows because $\eta_\xi$ is the image of the ergodic system $(X,\nu_\xi,T)$.

The measure $\eta_\xi$ is a joining of the Haar measures on $M_{i_1},\dots,M_{i_n}$ and need not be the Haar measure of $M_{\mathbf i}$, so we first identify it. Fix $\xi$ for which $\eta_\xi$ is ergodic. By the pointwise ergodic theorem for $\Z^d$ along the cubes $F_L:=\{-L,\dots,L\}^d$, applied to a countable dense subset of $C(M_{\mathbf i})$, $\eta_\xi$-a.e.\ point of $M_{\mathbf i}$ is generic for $\eta_\xi$ along $(F_L)_{L\geq 1}$; fix such a point $m=g\Gamma_{\mathbf i}$ with $g\in N_{\mathbf i}$. By Leibman's theorem on orbits of linear actions by translations \cite[\S1.4]{leibman-orbits} (see \cite{leibman-sequences} for a proof), there is a closed subgroup $E\subseteq N_{\mathbf i}$ containing $\beta_\xi(\Z^d)$ such that $\overline{\beta_\xi(\Z^d)m}=Em$ and the orbit $(\beta_\xi(t)m)_{t\in\Z^d}$ is well distributed in $Em$ with respect to its $E$-invariant probability measure $m_{Em}$, that is, along every F{\o}lner sequence of $\Z^d$. Comparing the averages along $(F_L)_{L\geq 1}$ shows that $\eta_\xi=m_{Em}$.

We check that this is a nilsystem in the sense of Definition~\ref{Definition: Bundles of nilsystems}. The stabiliser of $m$ in $E$ is the discrete subgroup $\Lambda:=E\cap g\Gamma_{\mathbf i}g^{-1}$. Since $E$ is $\sigma$-compact and acts transitively on the compact set $Em$, the orbit map induces a homeomorphism $E/\Lambda\cong Em$, so $\Lambda$ is cocompact in $E$. The lattice $\Gamma_{\mathbf i}$ is finitely generated, being cocompact in the compactly generated group $N_{\mathbf i}$; hence its subgroup $g^{-1}\Lambda g$ is finitely generated, as is every subgroup of a finitely generated nilpotent group. Writing $E=K\Lambda$ with $K\subseteq E$ compact, we see that $E$ is compactly generated. As a subgroup of $N_{\mathbf i}$, the group $E$ is nilpotent of step at most $k$, and since $\beta_\xi(\Z^d)\subseteq E$, the homomorphism $\beta_\xi$ acts on $E/\Lambda$ by left translations. Thus $(\supp\eta_\xi,\eta_\xi,\beta_\xi)\cong(E/\Lambda,m_{E/\Lambda},\beta_\xi)$ is an ergodic nilsystem of step at most $k$.

We next choose these nilsystem models measurably, using the countability of isomorphism classes of nilmanifolds (Theorem~\ref{t:countability}) and the Jankov--von Neumann uniformization theorem (Theorem~\ref{thm:JvN}).

\begin{claim}\label{Claim: measurable_trivialization}
Let $M=N/\Gamma$ be a nilmanifold of step at most $k$. Let $\mathcal S(M)\subseteq\Prob(M)\times\Hom(\Z^d,N)$ be the set of pairs $(\eta,\beta)$, where $\eta\in\Prob(M)$ is Haar measure on a subnilmanifold of $M$, $\beta\in\Hom(\Z^d,N)$ preserves $\eta$, and the induced action on $(\supp\eta,\eta)$ is ergodic; it is a Borel set, see the proof below. Then there exist a universally measurable partition $\mathcal S(M)=\bigsqcup_{j\in\N}\mathcal S_{j}$, indexed by the fixed enumeration $(N_j,\Gamma_j)_{j\in\N}$, $M_j=N_j/\Gamma_j$, of nilmanifolds of step at most $k$, and, for each $j$, universally measurable assignments
\[
(\eta,\beta)\in\mathcal S_{j}
\longmapsto
\theta_{\eta,\beta}:M_j\to M\quad(\text{defined }m_{M_j}\text{-a.e.})
\]
and
\[
(\eta,\beta)\in\mathcal S_{j}
\longmapsto
\widetilde\beta_{\eta,\beta}\in\Hom(\Z^d,N_j)
\]
such that $\theta_{\eta,\beta}$ is a measure-preserving isomorphism modulo null sets from $(M_j,m_{M_j})$ to $(\supp\eta,\eta)$ and
\[
\theta_{\eta,\beta}\bigl(\widetilde\beta_{\eta,\beta}(t)\cdot w\bigr)
=
\beta(t)\cdot \theta_{\eta,\beta}(w)
\]
for all $t\in\Z^d$ and $m_{M_j}$-a.e.\ $w\in M_j$.
The maps $(\eta,\beta,w)\mapsto\theta_{\eta,\beta}(w)$ are jointly universally measurable on $\mathcal S_j\times M_j$.
\end{claim}

\begin{proof}[Proof of the claim]
The preceding orbit-closure argument applies to any invariant ergodic probability measure for a nilrotation action on $M$. Thus $\mathcal S(M)$ is exactly the set of pairs $(\eta,\beta)$ for which $\eta$ is $\beta$-invariant and ergodic. Invariance is a Borel condition, as in Lemma~\ref{lem:Rn_borel}. To check ergodicity when $\eta$ varies, fix a countable uniformly dense family $(v_r)_{r\in\N}$ in $C(M)$ and let $A_L^\beta$ denote averaging over $F_L:=\{-L,\dots,L\}^d$. On invariant pairs, ergodicity is equivalent to
\[
\lim_{L\to\infty}\int_M
\left|A_L^\beta v_r-\int_M v_r\,d\eta\right|^2\,d\eta=0
\qquad\text{for every }r\in\N.
\]
This uses the mean ergodic theorem along $(F_L)_{L\geq 1}$ and the fact that a uniformly dense family in $C(M)$ is dense in $L^2(\eta)$ for every $\eta\in\Prob(M)$; uniform density, rather than density in $L^2$ of a fixed measure as in (R4), is needed here because $\eta$ varies. For fixed $r,L$, the integrand is jointly continuous in $(\eta,\beta,m)$ and $M$ is compact, so the integral is continuous in $(\eta,\beta)$. Hence $\mathcal S(M)$ is Borel.

\smallbreak
We encode an isomorphism and its inverse by a single graph joining. Using the fixed enumeration $(N_j,\Gamma_j)_{j\in\N}$, consider the relation
\[
R\subseteq\mathcal S(M)\times
\bigsqcup_{j\in\N}
\bigl(\Hom(\Z^d,N_j)\times\Prob(M_j\times M)\bigr).
\]
A tuple $((\eta,\beta),j,\alpha,\rho)$ belongs to $R$ if
\[
(\pr_{M_j})_*\rho=m_{M_j},\qquad
(\pr_M)_*\rho=\eta,
\]
if $(\alpha(t)\times\beta(t))_*\rho=\rho$ for every $t\in\Z^d$, and if the conditional measures of $\rho$ over both coordinates are almost surely Dirac. The marginal and invariance conditions are Borel as in Lemma~\ref{lem:Rn_borel}. For the two Dirac conditions, the proof of Lemma~\ref{lem2} applies verbatim, with $(X,\nu_\xi)$ replaced by $(M_j,m_{M_j})$ and by $(M,\eta)$, respectively, the conditional measures being given by Lemma~\ref{t:kallenberg}. In the second case the base measure $\eta$ varies, and one uses that $(\eta,p)\mapsto\int_M G(m,p)\,d\eta(m)$ is Borel for every bounded Borel function $G$. Thus $R$ is Borel.

Every section of $R$ is nonempty. Indeed, by the homogeneous description above, $(\supp\eta,\eta,\beta)$ is isomorphic to $(E/\Lambda,m_{E/\Lambda},\beta)$ for a closed subgroup $E\subseteq N$ containing $\beta(\Z^d)$ and a cocompact lattice $\Lambda\subseteq E$, with $E$ compactly generated and of step at most $k$. By Theorem~\ref{t:countability}, there are $j\in\N$ and an isomorphism of topological groups $\sigma:E\to N_j$ with $\sigma(\Lambda)=\Gamma_j$. It induces a measure-preserving homeomorphism $E/\Lambda\to M_j$ which intertwines $\beta$ with $\alpha:=\sigma\circ\beta\in\Hom(\Z^d,N_j)$, and the graph measure of the resulting isomorphism $(M_j,m_{M_j},\alpha)\to(\supp\eta,\eta,\beta)$ belongs to the required section. Theorem~\ref{thm:JvN} therefore gives universally measurable choices of $j$, $\widetilde\beta_{\eta,\beta}$ and $\rho_{\eta,\beta}$. The Dirac reconstruction of Lemma~\ref{lem2}, in the version just described and applied to both coordinates, gives $\theta_{\eta,\beta}$ and its inverse modulo null sets. This reconstruction is Borel in the graph measure and the spatial variable, and the composition of a Borel map with a universally measurable map is universally measurable, so the resulting maps are jointly universally measurable. Finally, invariance of the graph measure gives the asserted equivariance.
\end{proof}

Applying Claim~\ref{Claim: measurable_trivialization} to $M_{\mathbf i}$ and to the Borel map
\[
\xi\mapsto(\eta_\xi,\beta_\xi),
\]
which takes values in $\mathcal S(M_{\mathbf i})$ for $\mu$-a.e.\ $\xi\in\Omega_{\mathbf i}$ by the identification above, we obtain a countable $\mu$-measurable partition of $\Omega_{\mathbf i}$ into sets $\Omega_{\mathbf i,j}$ and, on each $\Omega_{\mathbf i,j}$, the fixed nilmanifold $M_j=N_j/\Gamma_j$ together with $\mu$-measurable homomorphisms
\[
\xi\mapsto \widetilde\beta_\xi\in\Hom(\Z^d,N_j)
\]
such that the fiber of $Z_n$ over $\xi$ is identified with $(M_j,m_{M_j})$ and the action becomes
\[
t\cdot(\xi,w)=(\xi,\widetilde\beta_\xi(t)\cdot w).
\]
Here $\mu$-measurability holds because preimages of universally measurable sets under Borel maps are universally measurable. By Remark~\ref{rem:borelmod0}, after discarding a $\mu$-null set we may take the selected indices, actions and graph measures to be Borel in $\xi$. The reconstruction in Lemma~\ref{lem2} then gives jointly Borel trivializations and their inverses on each piece.

The action is ergodic for $\mu$-a.e.\ $\xi$ by construction. Since the collection of tuples $\mathbf i$ and indices $j$ is countable, this gives a countable measurable partition of $\Omega$ over which $Z_n$ is represented as a bundle of $k$-step nilsystems. The bundle projection is the projection to $\Omega$, which agrees with $\pi_0$ because each $\pi_\ell$ lies over $\pi_0$.
\end{proof}

\subsection*{Taking the inverse limit}

\begin{lemma}\label{l3}
For every $r\in\mathbb N$,
\[
\lim_{n\to\infty}\big\| f_r-\E_\nu(f_r\mid \cA_n)\big\|_{L^2(\nu)}=0.
\]
\end{lemma}

\begin{proof}
Fix $r$ and take $n\ge r$. Since $\cA_n=\sigma(\pi_1,\dots,\pi_n)$ contains $\sigma(\pi_n)$,
\[
\|f_r-\E_\nu(f_r\mid \cA_n)\|_{2}\le \|f_r-\E_\nu(f_r\mid \sigma(\pi_n))\|_{2}.
\]
Disintegrating over $\Omega$ and using that $\pi_n$ restricts fiberwise to $\phi_{n,\xi}$ gives
\[
\|f_r-\E_\nu(f_r\mid \sigma(\pi_n))\|_2^2
=\int_\Omega \|f_r-\E_{\nu_\xi}(f_r\mid \phi_{n,\xi}^{-1}\cB(M_{j_n(\xi)}))\|_{L^2(\nu_\xi)}^2\,d\mu(\xi).
\]
But $(\xi;s_n(\xi))\in R_n$ implies that on $(X,\nu_\xi)$ the function $f_r$ is within $2^{-n}$ in $L^2$ of the factor $\phi_{n,\xi}^{-1}\cB(M_{j_n(\xi)})$, since (R3) gives a factor-measurable approximation and conditional expectation is the best $L^2$-approximation onto the factor. Thus the integrand is at most $2^{-2n}$ for $\mu$-a.e.\ $\xi$, and the claim follows.
\end{proof}

\begin{lemma}\label{l4}
Let $\cA_\infty:=\bigvee_{n\ge 1}\cA_n$. Then $\cA_\infty$ coincides with the full $\sigma$-algebra of $X$ modulo $\nu$-null sets.
\end{lemma}

\begin{proof}
Let $g\in L^2(X,\nu)$ be orthogonal to $L^2(\cA_\infty)$. Then $\E(g\mid \cA_n)=0$ for all $n$.

Fix $\varepsilon>0$ and choose $f_r\in\mathcal D$ such that $\|g-f_r\|_2<\varepsilon$.
By Lemma~\ref{l3}, choose $n\ge r$ such that
\[
\|f_r-\E(f_r\mid \cA_n)\|_2<\varepsilon.
\]
Now use the orthogonality of $g$ to $L^2(\cA_n)$: since $\E(f_r\mid \cA_n)\in L^2(\cA_n)$,
\[
\langle g,\, \E(f_r\mid \cA_n)\rangle
=
\langle \E(g\mid \cA_n),\, f_r\rangle
=0.
\]
Therefore,
\begin{align*}
\|g\|_2^2
&=\langle g,g\rangle\\
&=\langle g,\, g-f_r\rangle + \langle g,\, f_r-\E(f_r\mid \cA_n)\rangle + \langle g,\, \E(f_r\mid \cA_n)\rangle\\
&=\langle g,\, g-f_r\rangle + \langle g,\, f_r-\E(f_r\mid \cA_n)\rangle.
\end{align*}
By Cauchy--Schwarz,
\[
\|g\|_2^2
\le
\|g\|_2\,\|g-f_r\|_2 + \|g\|_2\,\|f_r-\E(f_r\mid \cA_n)\|_2
< 2\varepsilon\,\|g\|_2.
\]
If $\|g\|_2>0$, divide by $\|g\|_2$ to get $\|g\|_2<2\varepsilon$; since $\varepsilon>0$ was arbitrary, $\|g\|_2=0$.
Hence $g=0$, so $L^2(\cA_\infty)=L^2(X)$ and $\cA_\infty$ is the full $\sigma$-algebra modulo null sets.
\end{proof}

\noindent\emph{Conclusion of Theorem \ref{t:main}.}
The factors $Z_n$ form an increasing sequence of factors, and the natural maps $Z_{n+1}\twoheadrightarrow Z_n$ give an inverse system
\[
Z_1 \twoheadleftarrow Z_2 \twoheadleftarrow \cdots \twoheadleftarrow X.
\]
By Lemma~\ref{lem:join_bundle}, each $Z_n$ is a bundle of $k$-step nilsystems over $\Omega$. Lemma~\ref{l4} shows that $\bigvee_n \sigma(Z_n)$ is the full
$\sigma$-algebra of $X$ modulo null sets, hence $X$ is measurably isomorphic to the inverse limit $\varprojlim_n Z_n$.
This proves the reduced case $\mathrm X=\Zk_k(\mathrm X)$, and therefore, by the initial reduction, proves Theorem~\ref{t:main}.

\section{Proof of Theorem \ref{t:vertical_nilcharacters}}\label{s:vertical_nilcharacters}

We first check that every non-ergodic $k$-vertical nilcharacter on $\mathrm X$ belongs to $L^2(\Zk_k(\mathrm X))$. Write such a function as $f=F\circ\pi$, where $\pi:\mathrm X\to\mathrm Y$ is a factor map onto a bundle $\mathrm Y=\bigsqcup_{i\geq 0}\Omega_i\rtimes_{\alpha_i}M_i$ of nilsystems of step at most $k$. Since the fiber actions are ergodic for $\mu$-a.e.\ $\xi$, Fubini's theorem shows that every invariant function on $\mathrm Y$ depends only on $\xi$. Hence the ergodic decomposition of $\mathrm Y$ is $\xi\mapsto\delta_\xi\otimes m_{M_i}$ for $\xi\in\Omega_i$, and each ergodic component of $\mathrm Y$ is isomorphic to an ergodic nilsystem of step at most $k$. Such a nilsystem coincides with its own $k$th Host--Kra factor \cite{hk}. 
By Lemma~\ref{lem:Zk_fiberwise}, $\Zk_k(\mathrm Y)$ therefore coincides with the full $\sigma$-algebra of $\mathrm Y$ modulo $\delta_\xi\otimes m_{M_i}$ for $\mu$-a.e.\ $\xi$. Since $\Zk_k(\mathrm Y)$ contains the invariant factor, conditional expectations onto $\Zk_k(\mathrm Y)$ can be computed on the ergodic components, and hence $\mathrm Y=\Zk_k(\mathrm Y)$. By functoriality of Host--Kra factors (see, e.g., \cite{hkb}), $\pi^{-1}(\Zk_k(\mathrm Y))\subseteq\Zk_k(\mathrm X)$, so the pullback of $L^2(\mathrm Y)$ is contained in $L^2(\Zk_k(\mathrm X))$, and $f$ belongs to this space. Conversely, every non-ergodic $k$-vertical nilcharacter on $\Zk_k(\mathrm X)$ pulls back to one on $\mathrm X$. It therefore suffices to prove the spanning assertion when $\mathrm X=\Zk_k(\mathrm X)$.

By Theorem~\ref{t:main}, there is an inverse system of factors
\[
\mathrm{Y}_1 \twoheadleftarrow \mathrm{Y}_2 \twoheadleftarrow \cdots \twoheadleftarrow \mathrm{X}
\]
such that each $\mathrm{Y}_n$ is a bundle of nilsystems of step at most $k$ over the invariant factor $\Omega$, and $\mathrm{X}\cong \varprojlim_n \mathrm{Y}_n$ measure-theoretically. Hence
\[
L^2(\nu)=\overline{\bigcup_{n\geq 1}L^2(\mathrm{Y}_n)},
\]
where each $L^2(\mathrm{Y}_n)$ is identified with its pullback to $L^2(\nu)$. It is therefore enough to prove that, for every bundle $\mathrm Y$ of nilsystems of step at most $k$, the space $L^2(\mathrm{Y})$ is the closed linear span of the $k$-vertical nilcharacters on $\mathrm{Y}$.

Fix such a bundle
\[
\mathrm{Y}=\bigsqcup_{i\geq 0}\Omega_i\rtimes_{\alpha_i} M_i,
\qquad
M_i=N_i/\Gamma_i.
\]
Since the partition $(\Omega_i)_{i\geq 0}$ is countable, it suffices to work on a single component $\Omega_i\times M_i$. We suppress the index $i$ and write
\[
M=N/\Gamma,\qquad A:=N_{(k)}/(N_{(k)}\cap\Gamma).
\]

For each $\chi\in \widehat A$, define
\[
\mathcal H_\chi
:=
\Bigl\{
F\in L^2(\Omega_i\times M):
F(\xi,a\cdot m)=\chi(a)F(\xi,m)
\text{ for } \mu|_{\Omega_i}\otimes m_A\otimes m_M\text{-a.e. }(\xi,a,m)
\Bigr\}.
\]
Every element of $\mathcal H_\chi$, extended by $0$ outside $\Omega_i\times M$, is a $k$-vertical nilcharacter on $\mathrm Y$. Thus it remains to show that
\[
L^2(\Omega_i\times M)
=
\overline{\operatorname{span}}\bigcup_{\chi\in\widehat A}\mathcal H_\chi.
\]

Let $U_a$ denote the unitary action of $A$ on $L^2(\Omega_i\times M)$ given by
\[
(U_aF)(\xi,m):=F(\xi,a\cdot m).
\]
For $\chi\in\widehat A$, define
\[
P_\chi F
:=
\int_A U_aF\,\overline{\chi(a)}\,dm_A(a),
\]
where the integral is taken in $L^2(\Omega_i\times M)$. Then $P_\chi$ is the orthogonal projection onto $\mathcal H_\chi$. Indeed, for $b\in A$,
\begin{align*}
(P_\chi F)(\xi,b\cdot m)
&=
\int_A F(\xi,a b\cdot m)\,\overline{\chi(a)}\,dm_A(a)\\
&=
\chi(b)\int_A F(\xi,c\cdot m)\,\overline{\chi(c)}\,dm_A(c)\\
&=
\chi(b)(P_\chi F)(\xi,m),
\end{align*}
and the usual character orthogonality relations on the compact abelian group $A$ give
\[
P_\chi P_{\chi'}
=
\begin{cases}
P_\chi, & \chi=\chi',\\
0, & \chi\neq \chi'.
\end{cases}
\]

Now suppose that $F\in L^2(\Omega_i\times M)$ is orthogonal to every $\mathcal H_\chi$. Then
\[
P_\chi F=0
\qquad\text{for every }\chi\in\widehat A.
\]
Consider the continuous positive-definite function
\[
c:A\to\mathbb C,
\qquad
c(a):=\langle U_aF,F\rangle_{L^2(\Omega_i\times M)}.
\]
For every $\chi\in\widehat A$, its $\chi$-Fourier coefficient is
\[
\int_A c(a)\,\overline{\chi(a)}\,dm_A(a)
=
\left\langle
\int_A U_aF\,\overline{\chi(a)}\,dm_A(a),
F
\right\rangle
=
\langle P_\chi F,F\rangle
=
0.
\]
Since the characters of the compact abelian group $A$ span a dense subspace of $C(A)$, all Fourier coefficients of $c$ vanish, and hence $c=0$. In particular, $0=c(e)=\|F\|_2^2$, so $F=0$. Therefore the closed linear span of the spaces $\mathcal H_\chi$, $\chi\in\widehat A$, is all of $L^2(\Omega_i\times M)$.

Applying this argument on each component $\Omega_i\times M_i$ and summing over $i\geq 0$, we obtain that $L^2(Y)$ is the closed linear span of the $k$-vertical nilcharacters on $\mathrm Y$. Pulling these functions back through the factor maps $\mathrm X\to \mathrm Y_n$ proves the theorem.

\begin{remark}[Relative orthonormality]\label{rem:relative_orthonormal_vertical_nilcharacters}
The proof gives a slightly more structured statement on each bundle factor.
Let
\[
\mathrm{Y}=\bigsqcup_{i\geq 0}\Omega_i\rtimes_{\alpha_i} M_i,
\qquad M_i=N_i/\Gamma_i,
\]
be a bundle of nilsystems of step at most $k$, and put
\[
A_i:=A_k(M_i)=(N_i)_{(k)}/((N_i)_{(k)}\cap\Gamma_i).
\]
For $\chi\in\widehat{A_i}$, let
\[
\mathcal K_{i,\chi}
:=
\Bigl\{
h\in L^2(M_i):
h(a\cdot m)=\chi(a)h(m)
\text{ for }m_{A_i}\otimes m_{M_i}\text{-a.e. }(a,m)
\Bigr\}.
\]
The argument with the projections $P_\chi$ shows that
\[
L^2(M_i)=\bigoplus_{\chi\in\widehat{A_i}}\mathcal K_{i,\chi}.
\]
Choose, for each $i$ and each $\chi\in\widehat{A_i}$, an orthonormal basis
$(e_{i,\chi,\ell})_{\ell\geq 1}$ of $\mathcal K_{i,\chi}$. Then the functions
\[
F_{i,\chi,\ell}(\xi,m)
:=
1_{\Omega_i}(\xi)e_{i,\chi,\ell}(m),
\qquad \xi\in\Omega_i,\ m\in M_i,
\]
extended by zero outside $\Omega_i\times M_i$, are $k$-vertical nilcharacters
on $\mathrm{Y}$.

Moreover, the family $(F_{i,\chi,\ell})_{i,\chi,\ell}$ is relatively
orthonormal over the invariant factor $\Omega$ in the following sense:
\[
\E\bigl(F_{i,\chi,\ell}\overline{F_{i',\chi',\ell'}}\mid\Omega\bigr)=0
\]
unless $(i,\chi,\ell)=(i',\chi',\ell')$, while
\[
\E\bigl(|F_{i,\chi,\ell}|^2\mid\Omega\bigr)=1_{\Omega_i}.
\]
Furthermore, the closed linear span of all products
\[
F_{i,\chi,\ell}\psi,
\qquad
\psi\in L^\infty(\Omega),
\]
is all of $L^2(Y)$. Thus every bundle factor appearing in
Theorem~\ref{t:main} admits a relative orthonormal basis of $k$-vertical
nilcharacters, in direct analogy with the relative orthonormal basis of
eigenfunctions of Frantzikinakis and Host \cite[Section~5.2]{fh}.

Passing to the inverse limit in the proof of Theorem~\ref{t:vertical_nilcharacters}
therefore gives a countable family of non-ergodic $k$-vertical nilcharacters
on $\mathrm X$ whose $L^\infty(\Omega)$-linear span is dense in
$L^2(\Zk_k(\mathrm X))$. Since the relative orthonormal bases obtained from
different bundle factors in the inverse system need not be mutually relatively
orthogonal, we only use this totality statement for the inverse limit.
\end{remark}


\begin{thebibliography}{99}

\bibitem{abs}
E.~Ackelsberg, V.~Bergelson, and O.~Shalom.
\newblock Khintchine-type recurrence for 3-point configurations.
\newblock \emph{Forum Math. Sigma}, 10:\penalty0 e107, 57 pp., 2022.

\bibitem{ar}
E.~Ackelsberg and F.~K.~Richter.
\newblock An inverse theorem for sumsets of sets of positive density in the integers.
\newblock Preprint, \emph{arXiv}:2604.12864 [math.NT], 2026.

\bibitem{austin}
T.~Austin.
\newblock Extensions of probability-preserving systems by measurably-varying homogeneous spaces and applications.
\newblock \emph{Fundam. Math.}, 210\penalty0 (2):\penalty0 133--206, 2010.

\bibitem{bk}
H.~Becker and A.~S.~Kechris.
\newblock \emph{The descriptive set theory of Polish group actions}.
\newblock Lond. Math. Soc. Lect. Note Ser., Vol.~232, Cambridge University Press, Cambridge, 1996.


\bibitem{cz}
P.~Candela and B.~Szegedy,
\newblock {\em Nilspace factors for general uniformity seminorms, cubic exchangeability and limits},
\newblock Mem. Amer. Math. Soc. \textbf{1425}, American Mathematical Society, Providence, RI, 2023.

\bibitem{chuetal}
Q.~Chu, N.~Frantzikinakis, and B.~Host.
\newblock Ergodic averages of commuting transformations with distinct degree polynomial iterates.
\newblock \emph{Proc. Lond. Math. Soc. (3)}, 102\penalty0 (5):\penalty0 801--842, 2011.

\bibitem{de1937prevision}
B.~De~Finetti,
\newblock La pr{\'e}vision: ses lois logiques, ses sources subjectives,
\newblock {\em Ann. Inst. H. Poincar{\'e}} \textbf{7} (1937), 1--68.

\bibitem{edeko}
N.~Edeko.
\newblock On the isomorphism problem for non-minimal transformations with discrete spectrum.
\newblock \emph{Discrete Contin. Dyn. Syst.}, 39\penalty0 (10):\penalty0 6001--6021, 2019.

\bibitem{ejk}
N.~Edeko, A.~Jamneshan, and H.~Kreidler.
\newblock A Peter--Weyl theorem for compact group bundles and the geometric representation of relatively ergodic compact extensions.
\newblock \emph{Documenta Mathematica, to appear}.



\bibitem{f}
N.~Frantzikinakis.
\newblock Ergodicity of the Liouville system implies the Chowla conjecture.
\newblock \emph{Discrete Anal.}, 2017:\penalty0 41, 2017.

\bibitem{fh}
N.~Frantzikinakis and B.~Host.
\newblock Weighted multiple ergodic averages and correlation sequences.
\newblock \emph{Ergodic Theory Dyn. Syst.}, 38\penalty0 (1):\penalty0 81--142, 2018.


\bibitem{fk}
N.~Frantzikinakis and B.~Kuca.
\newblock Joint ergodicity for commuting transformations and applications to polynomial sequences.
\newblock \emph{Invent. Math.}, 239\penalty0 (2):\penalty0 621--706, 2025.

\bibitem{FraczykMachado}
M.~Fraczyk and S.~Machado,
\newblock {$\mathrm{ASL}_n(\mathbb Z)$} invariant random subsets of {$\mathbb Z^n$},
\newblock Preprint, arXiv:2605.16921 (2026).

\bibitem{hernandez2026}
F.~Hern{\'a}ndez.
\newblock Structure and spectrum of nonergodic nilsystems.
\newblock Preprint, \emph{arXiv}:2602.21021 [math.DS], 2026.


\bibitem{hk}
B.~Host and B.~Kra,
\newblock Nonconventional ergodic averages and nilmanifolds,
\newblock {\em Ann. of Math. (2)} \textbf{161} (2005), no.~1, 397-488.


\bibitem{hkb}
B.~Host and B.~Kra,
\newblock {\em Nilpotent Structures in Ergodic Theory},
\newblock Mathematical Surveys and Monographs, vol.~236,
\newblock American Mathematical Society, Providence, RI, 2018.

\bibitem{HoweMoore}
R.~E.~Howe and C.~C.~Moore,
\newblock Asymptotic properties of unitary representations,
\newblock {\em J. Funct. Anal.} \textbf{32} (1979), no.~1, 72--96.

\bibitem{hsy}
W.~Huang, S.~Shao, and X.~Ye.
\newblock The polynomial Furstenberg joining and its applications.
\newblock \emph{Sci. China, Math.}, 69\penalty0 (1):\penalty0 93--166, 2026.


\bibitem{ke}
A.~S.~Kechris,
\newblock {\em Classical Descriptive Set Theory},
\newblock Graduate Texts in Mathematics, vol.~156,
\newblock Springer, 1995.

\bibitem{ka}
O.~Kallenberg,
\newblock Stationary and invariant densities and disintegration kernels,
\newblock {\em Probab. Theory Relat. Fields} \textbf{160} (2014), no.~3-4, 567-592.

\bibitem{leibman-factor}
A.~Leibman,
\newblock Host--Kra and Ziegler factors and convergence of multiple averages,
\newblock in {\em Handbook of Dynamical Systems, Vol.~1B},
B.~Hasselblatt and A.~Katok, eds.,
Elsevier, 2005, pp.~841--853.


\bibitem{leibman-sequences}
A.~Leibman,
\newblock Pointwise convergence of ergodic averages for polynomial sequences of translations on a nilmanifold,
\newblock {\em Ergodic Theory Dynam. Systems} \textbf{25} (2005), no.~1, 201--213.

\bibitem{leibman-orbits}
A.~Leibman,
\newblock Pointwise convergence of ergodic averages for polynomial actions of $\mathbb Z^d$ by translations on a nilmanifold,
\newblock {\em Ergodic Theory Dynam. Systems} \textbf{25} (2005), no.~1, 215--225.


\bibitem{Marklof}
J.~Marklof,
\newblock Kinetic limits of dynamical systems,
\newblock {\em Hyperbolic dynamics, fluctuations and large deviations}, Proc. Sympos. Pure Math., vol.~89, Amer. Math. Soc., Providence, RI, 2015, pp.~195--223.


\bibitem{z}
T.~Ziegler,
\newblock Universal characteristic factors and {F}urstenberg averages,
\newblock {\em J. Amer. Math. Soc.} \textbf{20} (2007), no.~1, 53-97. 

\bibitem{zk}
P.~Zorin-Kranich.
\newblock A double return times theorem.
\newblock \emph{Isr. J. Math.}, 229\penalty0 (1):\penalty0 255--267, 2019.



\end{thebibliography}
\end{document}